\documentclass[twoside, 11pt]{amsart}

\usepackage[T1]{fontenc}
\usepackage[utf8]{inputenc}
\usepackage[english]{babel}
\usepackage{indentfirst}
\usepackage{amsmath, amssymb, amscd, amsthm, amsfonts}
\usepackage{amssymb, amsfonts}
\usepackage{amsthm, amscd}

\usepackage{inputenc}
\usepackage{mathtools}
\usepackage{hyperref}
\usepackage{verbatim}
\usepackage{color}
\usepackage{enumerate}
\usepackage{mathrsfs}

\usepackage{tikz}
\usetikzlibrary{3d}
\usetikzlibrary{hobby}
\usetikzlibrary{arrows}
\usetikzlibrary{calc}
\usetikzlibrary{patterns}
\usepackage{pgfplots}
\usetikzlibrary{intersections, pgfplots.fillbetween}

\theoremstyle{definition}
\newtheorem{definition}{Definition}

\theoremstyle{plain}
\newtheorem{theorem}[definition]{Theorem}

\theoremstyle{plain}
\newtheorem{conjecture}[definition]{Conjecture}

\theoremstyle{plain}

\theoremstyle{plain}

\theoremstyle{plain}

\theoremstyle{definition}
\newtheorem{remark}[definition]{Remark}
\theoremstyle{plain}
\newtheorem{lemma}[definition]{Lemma}
\theoremstyle{plain}
\newtheorem{proposition}[definition]{Proposition}
\theoremstyle{plain}
\newtheorem{corollary}[definition]{Corollary}
\theoremstyle{definition}

\theoremstyle{definition}

\DeclareMathOperator{\hess}{Hess}

\DeclareMathOperator{\di}{div}

\newcommand{\R}{\mathbb{R}}

\newcommand{\N}{\mathbb{N}}

\newcommand{\lef}{\left(}
\newcommand{\rig}{\right)}

\DeclareMathOperator{\conv}{Conv}

\newcommand{\beq}{\begin{equation}}
\newcommand{\eeq}{\end{equation}}

\DeclareMathOperator{\dist}{dist}

\begin{document} 

\date{\today}

\address{Francesco Chini, Department of Mathematical Sciences, Copenhagen University.}
\email{chini@math.ku.dk}

\title{Simply connected translating solitons contained in slabs}
\thanks{The author was partially supported by the Villum Foundation's QMATH Centre.}

\author{Francesco Chini\\}

\maketitle

\begin{abstract}
In this work we show that $2$-dimensional, simply connected,  translating solitons of the mean curvature flow embedded in a slab of $\R^3$ with entropy strictly less than $3$ must be mean convex and thus, thanks to a result by J. Spruck and L. Xiao \cite{sx17},  are convex.
 Recently, such $2$-dimensional convex translating solitons have been completely classified in \cite{himw19a}, up to an ambient isometry, as  vertical plane, (tilted) grim reaper cylinders, $\Delta$-wings and bowl translater. These are all contained in a slab, except for the rotationally symmetric bowl translater. New examples by \cite{hmw19} show that the bound on the entropy is necessary.
\end{abstract}

\section*{Introduction}

In \cite{br16}, Brendle proved that any properly embedded $2$-dimensional mean curvature self-shrinker in $\R^3$ which is homeomorphic to an open subset of the sphere must be a round sphere, or a cylinder or a plane, solving two problems posed by Ilmanen (see 14 and 15 in \cite{il03}). In particular, it follows that the round sphere is the only closed, embedded shrinker with genus $0$. The main step in Brendle's paper was to first prove that any shrinker satisfying such a topological assumption must be mean convex and with polynomial area growth (his argument was partially inspired by \cite{ro95}). Then the conclusion follows from Theorem 10.1 in \cite{cm12}, which is a refinement of a theorem by Huisken  \cite{hu93} (see also \cite{hu90} for the closed case).

One cannot expect such a strong result  for translating solitons, even under the more restrictive topological assumption of being simply connected. In fact Hoffman, Mart\'in and White \cite{hmw19} recently constructed new examples of properly embedded translaters which are simply connected but are not mean convex. The most surprising one is the so-called \emph{pitchfork} translater which has entropy equal to $  3$ and is contained in a slab.

In these notes we will consider smooth $2$-dimensional translating solitons of the mean curvature flow in $\R^3$, i.e. smooth surfaces immersed in $\R^3$ satisfying the equation
\begin{equation}\label{translating_equation}
H = - \langle \nu, e_3 \rangle,
\end{equation}
where $\nu$ is a smooth unit normal vector field on $\Sigma$ and $H$ denotes the mean curvature. Note that $\Sigma$ satisfies \eqref{translating_equation} if and only if the 1-parameter family $\Sigma + t e_3$ is a mean curvature flow, with $t \in \R$. 

Following \cite{cm12} denote the entropy of $\Sigma$ by $\lambda(\Sigma)$ (see Appendix \ref{appendix} for details).

The main contribution of this work is the following result.

\begin{theorem}\label{theorem_slab}
Let $\Sigma^2 \subseteq \R^3$ be a complete, embedded, translater satisfying the following assumptions:

\begin{enumerate}[(i)]
\item $\Sigma$ is simply connected,
\item $\lambda(\Sigma) < 3$,
\item $\Sigma$ is contained in a slab. 
\end{enumerate}

Then $\Sigma$ is mean convex.
\end{theorem}

Spruck and Xiao proved that $2$-dimensional, mean convex translaters are actually convex (Theorem 1.1 in \cite{sx17}). Therefore their result together with the classification of Hoffman, Ilmanen, Mart\'in and White \cite{himw19a}, yields the following corollary. 

\begin{corollary}
Let $\Sigma$ be as in Theorem \ref{theorem_slab}. Then $\Sigma$ is, up to an ambient isometry, one of the following translating solitons:
\begin{enumerate}[(i)]
\item a vertical plane,
\item a grim reaper cylinder (possibly tilted),
\item a $\Delta$-wing translater. 
\end{enumerate}
\end{corollary}

\begin{remark}
As mentioned above, the pitchfork example shows that the bound on the entropy in the assumptions of Theorem \ref{theorem_slab} is necessary and cannot be relaxed. 

On the other hand, there are currently no known examples of complete translaters contained in a slab with entropy strictly less than $3$ which are not simply connected. So it is not clear whether the topological assumption is necessary.
Hershkovits \cite{he18} classified translaters with entropy less or equal than the entropy of a cylinder without any further assumptions. More precisely, he proved that a translater $\Sigma^2 \subseteq \R^3$ satisfying the following entropy bound 
\begin{equation}\label{her_entropy_bound}
\lambda(\Sigma) \le \lambda(\mathbb{S}^1 \times \R) = \sqrt{\frac{2 \pi}{e}} \approx 1.52
\end{equation}
must be either a plane ($\lambda(\Sigma) = 1$) or the rotationally symmetric bowl translater ($\lambda(\Sigma) = \sqrt{\frac{2 \pi}{e}}$). However, even though Hershkovits does not need any toplogical assumption, his bound \eqref{her_entropy_bound} is much more restrictive than the entropy bound in our Theorem \ref{theorem_slab}. In a later work, Hershkovits, Haslhofer and Choi  \cite{hhc18} completely classified $2$-dimensional ancient mean curvature flows with entropy less or equal to $\lambda(\mathbb{S}^1 \times \R)$ and they used this classification to prove the mean convex neighborhood conjecture (see also the very recent paper \cite{hhcw19} for a higher dimensional analog).

Moreover, we believe that the assumption $(iii)$ in Theorem \ref{theorem_slab} is purely technical and  can be removed. 
\end{remark}

\begin{conjecture}\label{low_entropy_theorem}
Let $\Sigma^2 \subseteq \R^{3}$ be an embedded simply connected translater such that $\lambda(\Sigma) <3$. 

Then $\Sigma$ is mean convex. 
\end{conjecture}

\begin{remark}\label{remark_stable_trans}
Simply connected translating solitons are particularly interesting because  it is known (see \cite{ir17}, \cite{ir19}, \cite{ks18}) that complete $2$-\!\! dimensional stable translaters in $\R^3$ must be simply connected. By \emph{stable translaters} we mean translaters which are linearly stable as minimal surfaces w.r.t. to the conformally flat metric $e^{x_3} \delta_{ij}^{\text{Eucl}}$ (for more details see the survey \cite{himw19b}). 

Observe that for shrinking solitons there is a connection between \emph{stability} and \emph{mean convexity}. 
More precisely, Colding and Minicozzi \cite{cm12} proved that entropy stable shrinkers (with polynomial volume growth) are mean convex. Entropy stability is intimately related with the index of the stability operator of shrinkers as minimal surfaces in the gaussian metric $e^{-\frac{|x|^2}{4}}\delta^{\text{Eucl}}_{ij}$. For these reasons and motivated by Theorem \ref{theorem_slab} , we are tempted to state the following conjecture.  
\end{remark}

\begin{conjecture}
Let $\Sigma^2 \subseteq \R^3$ be a complete embedded stable translater. Then $\Sigma$ is mean convex and therefore, thanks to \cite{sx17} and \cite{himw19a}, up to an ambient isometry,  one of the following translating solitons:
\begin{enumerate}[(i)]
\item a vertical plane,
\item a grim reaper cylinder (possibly tilted),
\item a $\Delta$-wing translater,
\item the rotationally symmetric bowl translater.
\end{enumerate}
\end{conjecture}



\subsection*{Organization of the paper}

In Section~\ref{section_curvature_estimate} we derive a curvature estimate for embedded, simply connected translating solitons with  finite entropy, which allows us to use a compactness theorem (based on a standard Arzelà-Ascoli argument) in a crucial step of the proof of Theorem \ref{theorem_slab}. The curvature estimate is a consequence of an estimate by Schoen and Simon \cite{ss83}.

In Section~\ref{section_projection}, which is the longest section of this work, we prove Theorem~\ref{new_refinement}, which is a refinement of results contained in the paper by M\o ller and the author \cite{cm19}.  The proofs are based on a combination of the Omori-Yau maximum principle and barrier arguments. As a byproduct of Theorem \ref{new_refinement}, we obtain a Bernstein type theorem for $1$-periodic properly immersed translaters. 

Section \ref{section_Z} is devoted to the study of the structure of the intersection $Z \coloneqq \Sigma \cap T_p \Sigma$, where $T_p \Sigma$ denotes the geometric tangent space of $\Sigma$ at some point $p \in \Sigma$ where $H(p) = 0$. This is done by observing that $Z$ is the nodal set of a function $f \colon \Sigma \to \R$ solving an elliptic PDE of the kind $\Delta^\Sigma f = h f$ for some function $h \in C^{\infty}(\Sigma)$ and applying a result by \cite{ch76}. Under the assumption of $\Sigma$ being simply connected, we study the topology of $Z$. Namely, we show, by using a maximum principle argument, that each connected component of $Z$ is contractible.

In Section \ref{section_H} we  study the structure of $\{H = 0\}$ and show that on a translater the unit normal vector cannot be constant along $\{H = 0\}$ unless the translater is mean convex. 


In Section~\ref{section_proof} we finally prove Theorem~\ref{theorem_slab}.  The proof   proceeds by contradiction. We assume that $\Sigma$ is not mean convex and we  carefully study the intersection $Z = \Sigma \cap T_p \Sigma$, where $p \in \Sigma$ is some point such that $H(p) = 0$.   We  distinguish two different cases and we see that one case contradicts the entropy bound and the other one contradicts the topological assumption of $\Sigma$ being simply connected. 

In the Appendix \ref{appendix} we recall the definition and some basic well-known properties of the entropy functional.


\subsection*{Acknowledgments:} I would like to thank Toby Colding and Bill Minicozzi for their hospitality at MIT during Spring 2019, where some of the ideas contained in this work were developed. During my stay, I had the opportunity to have many enlightening conversations with several people, but I am mostly in debt  with Kyeongsu Choi, from whom I have learned many things.

I want to express my gratitude also to Alexander Friedrich and Felix Lubbe for reading a preliminary version of this paper and giving me useful  feedbacks. 

I would like to thank my supervisor Niels Martin M\o ller for his constant support. 


\section{Curvature estimate}\label{section_curvature_estimate}\label{section_1}

In this section we derive a curvature estimate for simply connected translating solitons with finite entropy, which is of independent interest. Similar results have already been obtained for $2$-dimensional translating solitons, but under different assumptions. See for instance Theorem 3.2 in \cite{sh15}, Theorem 4.8 in \cite{guang16},  Theorem 2.8 in \cite{sx17} and Theorem A.3 in \cite{hmw19}.

\begin{proposition}\label{curvature_estimate}
Let $\Sigma^2 \subseteq \R^3$ be a complete, embedded, simply connected translater such that $\lambda({\Sigma}) < \infty$.

Then there exists a constant $C > 0$ such that $|A| \le C$.
\end{proposition}

\begin{proof}
Remark \ref{remark_entropy_area_growth} in the Appendix \ref{appendix}, implies that  there exists a constant $C_1 = C_1(\lambda(\Sigma)) > 0$ such that
$$
\text{Area}(\Sigma \cap \mathcal{B}_R(x)) \le C_1 R^2
$$  
 for any radius $R > 0$, for any point $x \in \R^3$, where $\mathcal{B}_R(x)$ is the open ball in $\R^3$ centered at $x \in \R^3$ of radius $R > 0$.

 Recall that  $\Sigma^2 \subseteq \R^3$ is said to have $(\gamma_1, \gamma_2)$-quasiconformal Gauss map, with $\gamma_1, \gamma_2 \ge 0$, if 
\begin{equation}\label{eq_quasiconformal}
|A|^2(p) \le -\gamma_1 K(p) + \gamma_2, \qquad p \in \Sigma,
\end{equation}
where $K$ denotes the Gauss curvature. Since $\Sigma$ is a translater, i.e. satisfies \eqref{translating_equation}, then it has $(2, 1)$-quasiconformal Gauss map, namely
$$
|A|^2(p) = - 2 K(p) + H^2(p) \le -2 K(p) + 1, \qquad p \in \Sigma.
$$

We can apply the estimate for embedded simply connected surfaces with quasiconformal Gauss map of Schoen and Simon   \cite{ss83}.  More precisely, let us fix $R>0$.  Theorem~1 in \cite{ss83} implies that  there exist  constants $C_2 = C_2(R, \lambda(\Sigma)) > 0$ and $\alpha = \alpha(R, \lambda(\Sigma)) \in (0,1)$ such that for any $p \in \Sigma$ we have
\begin{equation}\label{holder_gauss_map}
\| \nu(x) - \nu(\tilde{x})\| \le C_2 \|x - \tilde{x}\|^{\alpha},
\end{equation}
for any $x, \tilde{x} \in \Sigma'$, where $\Sigma'$ is the connected component of $\Sigma \cap \mathcal{B}_R(p)$ containing $p$. 

Equation \eqref{holder_gauss_map} implies that there exists $\varrho = \varrho(\lambda(\Sigma)) > 0$, such that for any $p \in \Sigma$, the connected component of $\Sigma \cap \mathcal{B}_\varrho(p)$ containing $p$ is the graph of a smooth function $u$ over an open domain $\Omega$ of $T_p\Sigma$ such that $\|\nabla u\| <1 $. Note that $T_p\Sigma$ denotes the geometric tangent plane of $\Sigma$ at $p$. It is easy to see that the $2$-dimensional disk $B_{\frac{\varrho}{\sqrt{2}}}(p) \subseteq T_p \Sigma$ is contained in $\Omega$. With a small abuse of  notation, we keep denoting the restriction $u|_{B_{\frac{\varrho}{\sqrt{2}}}(p)}$ by $u$. Note that we have 
$
\sup_{B_{\frac{\varrho}{\sqrt{2}}}(p)} |u| \le \frac{\varrho}{\sqrt{2}}.
$
We summarize the above observations as follows
\begin{equation}\label{c1_bound_u}
\| u\|_{C^{1}\lef B_{\frac{\varrho}{\sqrt{2}}}(p)\rig} \le 1 + \frac{\varrho}{\sqrt{2}}.
\end{equation}

Observe that equation \eqref{translating_equation} implies that $u$ solves the following elliptic equation
\begin{equation}\label{PDE_u}
\di\lef \frac{D u }{\sqrt{1 + \|Du\|^2}} \rig = F,
\end{equation}
where $F(y) \coloneqq  \langle \nu(y, u(y)), e_3 \rangle$ is a smooth function and $y= (y_1, y_2)$ are Cartesian coordinates on the plane $T_p\Sigma$. Observe that $|F| \le 1$ and from \eqref{holder_gauss_map}, we have a uniform estimate of the $\alpha$-H\"older norm of $F$. Namely, given $y, \tilde{y} \in B_{\frac{\varrho}{\sqrt{2}}}(p) \subseteq T_p \Sigma$, we have
\begin{align*}
\frac{|F(y) - F(\tilde{y})|}{\|y - \tilde{y}\|^\alpha}  &\le \frac{\|\nu(y, u(y)) - \nu(\tilde{y}, u(\tilde{y}))\|}{\|y - \tilde{y}\|^\alpha}\\
&\le 2^\alpha \frac{\|\nu(y, u(y)) - \nu(\tilde{y}, u(\tilde{y}))\|}{\|(y, u(y)) - (\tilde{y}, u(\tilde{y})\|^\alpha}\\
&\le 2^\alpha \frac{C_2}{R^\alpha} \eqqcolon C_3.
\end{align*}

We can think of \eqref{PDE_u} as a linear elliptic equation in $u$ where the coefficients depend on $Du$. The  uniform $C^1$ estimate \eqref{c1_bound_u} implies uniform ellipticity and a uniform bound on $C^1$-norms of the coefficients. This, together with the uniform estimate of the $\alpha$-H\"older norm of $F$, allow us to apply standard Schauder estimates (see for instance Corollary 6.3 in \cite{gt01}). Therefore for every $\delta \in (0, \frac{\varrho}{\sqrt{2}})$ there exists a constant $C_4>0$ such that
\begin{equation}
\|u\|_{C^2\lef B_{\delta}(p)\rig} \le C_4.
\end{equation}
The constant $C_4$ depends only on $\delta$ and on the bounds on the $C^1$-norm of $u$ and the $\alpha$-H\"older norm of $F$. Observe that none of those bounds depend on the point $p \in \Sigma$. In fact, they ultimately depend on the value of the entropy $\lambda(\Sigma)$. This concludes the proof, since $|A|^2(p) = |\hess u |^2(p)$.
\end{proof}

\begin{remark}
Note that in the proof of Proposition \ref{curvature_estimate}, after using \cite{ss83} to show that $\Sigma$ can be  locally described as a graph with a uniform control on its $C^1$ norm, we could have obtained a uniform estimate for $|A|$ by applying the local curvature estimate by Ecker and Huisken, i.e. Theorem 3.1 in \cite{eh91}. 
\end{remark}


\section{Asymptotic behavior  of properly immersed translaters}\label{section_projection}

In this section, $\Sigma^2 \subseteq \R^3$ is a properly immersed translater. We do not assume any bound on the entropy and we do not put any restriction on  the topology of $\Sigma$. 

Let us fix some notation. 
\begin{itemize}
\item $\pi \colon \R^3 \to \R^2$  denotes the projection $\pi(x_1, x_2, x_3) = (x_1, x_2)$. 
\item $\conv(\cdot)$ denotes the (closed) convex hul.
\item $B_\varrho (q)$ denotes the open ball in $\R^2$ centered at a point $q \in \R^2$, with radius $\varrho > 0$.
\item We say that a plane $P \subseteq \R^3$ is \emph{vertical} if $P \parallel e_3$. 
\item We say that a halfspace $\mathcal{H} \subseteq \R^3$ is \emph{vertical} if the plane $\partial \mathcal{H}$ is vertical.
\end{itemize}

\begin{remark}\label{remark_asymptotic_planes}
 From \cite{cm19} (see also the more general case of ancient flows \cite{cm19b}) it is known that $\conv(\pi(\Sigma))$ is either a line, a strip, a half-plane or the whole $\R^2$. Therefore $\pi^{-1} (\partial \conv (\pi(\Sigma)))$ can be, respectively, only one of the following 
\begin{enumerate}[(i)]
\item a vertical plane,
\item two parallel vertical planes,
\item the empty set.
\end{enumerate}
We will see in this section that, if we are in Case (i) or (ii), $\Sigma$ is (in some weak sense) asymptotic to $\pi^{-1} (\partial \conv (\pi(\Sigma)))$ as $x_3 \rightarrow \infty $. See the  Theorem~\ref{new_refinement} and Corollary~\ref{corollary_of_proposition_refinement} below for a precise statement.
\end{remark}

\begin{theorem}\label{new_refinement}
Let $\Sigma^2 \subseteq \R^3$ be a properly immersed translater such that $\partial\conv(\pi(\Sigma)) \ne \emptyset$. 

Then for every $q \in \partial\conv(\pi(\Sigma))$ and for every $\varrho > 0$ we have that
\begin{equation}\label{eq_prop_refinement}
\sup_{\Sigma \, \cap \, \pi^{-1}(B_\varrho (q))} x_3 = +\infty. 
\end{equation}
\end{theorem}

\begin{proof}
Let us assume for contradiction that there exists $q^*\in \partial\conv(\pi(\Sigma))$ and a radius $\varrho^*>0$ such that 

\begin{equation*}
\sup_{\Sigma \, \cap \, \pi^{-1}({B_{2\varrho^*}(q^*)})} x_3 < +\infty.
\end{equation*} 

Up to a translation in the $e_3$ direction, we can assume that 
\begin{equation}\label{assumption_prop_new_refinement}
\sup_{\Sigma \, \cap \, \pi^{-1}(\overline{B_{\varrho^*}(q^*)})} x_3 < 0,
\end{equation} 
where $\overline{B_{\varrho^*}(q^*)}$ is the closure of $B_{\varrho^*}(q^*)$. 

W.l.o.g.\!\! we can assume that $\pi(\Sigma)$ is contained in the upper half-plane of $\R^2$, i.e. $\pi(\Sigma) \subseteq \{(x_1, x_2) \in \R^2 \colon x_2 \ge 0\}$ and let us assume that the $x_1$-axis $\{(x_1, x_2) \in \R^2 \colon  x_2 = 0\} $ is a connected component of $\partial\conv(\pi(\Sigma))$. Let us also assume $q^* = (0, 0)$.

The rest of the proof will be divided into three steps. 
\begin{enumerate}[(i)]
\item By using the Omori-Yau maximum principle for properly immersed translaters, we are going to prove that $x_3$ is bounded from above  on $\Sigma\, \cap \, \pi^{-1}(\mathcal{K})$, for every compact set $\mathcal{K} \subseteq \{(x_1, x_2) \in \R^2 \colon 0 \le x_2 < \varrho^*\}$. 
\item By using a family of grim reaper cylinders as barriers, we will prove that $x_3$ is uniformly bounded from above on $$\Sigma_\delta \coloneqq \Sigma \cap \{x \in \R^3 \colon 0 \le x_2 \le \delta \},$$ for every $\delta < \varrho^*$.
\item By using a family of $\Delta$-wing translaters as barriers, we will finally get a contradiction by proving that $\Sigma_\delta = \emptyset$, for every $\delta < \varrho^*$.
\end{enumerate}

\textbf{Step (i)}: Observe that for every $s \in \R$ such that $|s| > \varrho^*$, there exists a unique closed disk $D_s \subseteq \{(x_1, x_2) \colon x_2 \ge 0\}$ such that $D_s$ is tangent to the $x_1$-axis at the point $(s, 0)$, i.e. 
$$
D_s \cap \{(x_1, x_2) \in \R^2 \colon x_2 = 0\} = \{(s, 0)\}
$$
 and such that  $(0, \varrho^*) \in \partial D_s$. See Figure \ref{fig_1}. 

\begin{center}
\begin{figure}
\begin{tikzpicture}[use Hobby shortcut]
\clip (-6,3) rectangle + (12,-4.5);
\draw[->, name path =B] (- 3, 0) -- (0, 0) ;    
\draw[  name path = A] (-3,5) circle (5);
\tikzfillbetween[of=A and B]{blue, opacity=0.1};
\draw[fill = white] (-3,5) circle (5);
\draw[ fill=white, opacity = 1] (0,0) circle (1);
\draw[] (-3,5) circle (5);
\draw[->] (0, -1.5) -- (0,3);
\draw (0, 2.8) node [right] {$x_2$};
\draw[->] (- 6, 0) -- (3, 0) ;    
\draw (3, 0) node [below] {$x_1$};
\draw (-2.5, 1.5) node {$D_s$};
\fill (0,0) circle[radius=1pt] node[below right] {$q^* \! =\! (0, 0)$};
\fill (0,1) circle[radius=1pt] ;
\draw (0.2, 1.1) node [right] {$(0, \varrho^*)$};
\fill (-3,0 ) circle[radius=1pt] node[below] {$(s, 0)$};
\draw (-1.3, 0.1) node {$\mathcal{A}_s$};
\end{tikzpicture}
\caption{} \label{fig_1}
\end{figure}
\end{center}

Observe that $\{(x_1, x_2) \in \R^2 \colon x_2 \ge 0\} \setminus \lef D_s \cup B_{\varrho^*}(0)\rig$ has two connected components a bounded one and an unbounded one. Let us call $\mathcal{A}_s$ the bounded one. Observe that the family $\lef\mathcal{A}_s\rig_{|s|> \varrho^*}$ together with $B_{\varrho^*}(0)$, cover the strip $\{(x_1, x_2) \colon 0 \le x_2 < \varrho^*\}$. Namely, 
\begin{equation}\label{cover_strip_A}
\lef B_{\varrho^*}(0) \, \cup\! \bigcup_{s \in \R \colon |s| > \varrho^*} \mathcal{A}_s \rig \supseteq \{(x_1, x_2) \colon 0 \le x_2 < \varrho^*\}.
\end{equation}
We are now going to prove that $x_3$ is bounded on $\Sigma \cap \pi^{-1}(\mathcal{A}_s)$ for every $s \in \R$ such that $|s| > \varrho^*$. This will finish the proof of Step (i), because of \eqref{cover_strip_A}. We will do this by using the Omori-Yau maximum principle for properly immersed translaters and we refer to \cite{cm19} and to \cite{xi15} for details. 

Let us assume for contradiction that there exists $s^* \in \R$ such that $|s^*| > \varrho^*$ and such that 
$$
\sup_{\Sigma \, \cap \, \pi^{-1}(\mathcal{A}_{s^*})} x_3 = + \infty.
$$
 Let  $c \in \{(x_1, x_2) \colon x_2 \ge 0\}$ be the center of the disk $D_{s_0}$ and let $R > 0$ be its radius. Let $\mathcal{L}$ be the vertical line passing through the center $c$, i.e. $\mathcal{L} \coloneqq \pi^{-1}(\{c\}) = \{c\} \times \R$. Let us define a function $f \colon \R^3 \to \R$ as follows:
\begin{equation} 
f(x) \coloneqq
\begin{cases}
\dist \lef x, \mathcal{L}\rig  \qquad &\text{ if } \pi(x) \in \mathcal{A}_{s^*}\\
R &\text{ if } \pi(x) \notin \mathcal{A}_{s^*}.
\end{cases}
\end{equation}

Since $\mathcal{A}_{s^*}$ is bounded, $f$ is bounded. 
Observe that the the set of points where  $f|_{\Sigma}$  may be discontinuous is $\pi^{-1} \lef \partial B_{\varrho^*}(q^*) \cap  \partial \mathcal{A}_{s^*} \rig \cap \Sigma$, which is contained in $\pi^{-1}\lef \partial B_{\varrho^*}(q^*) \rig \cap \Sigma$. 
Let us consider the translater with boundary 
$$
\widetilde{\Sigma} \coloneqq \Sigma \cap \{x \in \R^3 \colon x_3 \ge 0\}.
$$

From \eqref{assumption_prop_new_refinement}, we have that 
$$
\pi^{-1}\lef \partial B_{\varrho^*}(q^*) \rig \cap \widetilde{\Sigma}  = \pi^{-1}\lef \partial B_{\varrho^*}(q^*) \rig \cap \Sigma \cap \{x \in \R^3 \colon x_3 \ge 0 \} = \emptyset,
$$
therefore $f|_{\widetilde{\Sigma}}$ is continuous. Moreover, $f|_{{\Sigma} \cap \pi^{-1}(\mathcal{A}_{s^*})}$ is smooth.
From standard computations (see \cite{cm19}) and using equation \eqref{translating_equation}, one can easily see that on $\Sigma \cap \pi^{-1}(\mathcal{A}_{s^*})$
\begin{equation}\label{omori_yau_eq_f} 
\Delta^{\Sigma} f = \frac{1 - \| \nabla^\Sigma f\|^2}{f} - \langle \nabla^{\R^3} f, \nu\rangle \langle \nu, e_3 \rangle.
\end{equation}

As in the proof of Theorem 1.2 in \cite{cm19}, we will use the Omori-Yau maximum principle combined with an ``adiabatic trick''. More precisely, we would like to apply the Omori-Yau maximum principle to the function $f|_{\widetilde{\Sigma} }$ defined on the translater with boundary $\widetilde{\Sigma }$.  But we need to employ the adiabatic trick because the maximum might be reached on the boundary $\partial \widetilde{\Sigma} = \Sigma \cap \{x_3 = 0\}$.

Let $\psi \colon \R \to \R$ be a smooth cut-off function such that 
\begin{itemize}
\item $0 \le \psi \le 1$,
\item $\psi|_{(-\infty, 0]} \equiv 0$,
\item $\psi|_{[1, \infty)} \equiv 1$.
\end{itemize}
For every $l >0$, let  $\chi_l \colon \R^3 \to \R$ be a function defined as follows:
$$
\chi_l(x) \coloneqq \psi\lef \frac{x_3}{l} \rig.
$$
Observe that there exists a constant $C$, which does not depend on $l$, such that 
\begin{equation}\label{estimates_chi} 
\sup_{x \in \R^3} \| \nabla^{\R^3} \chi (x) \| \le \frac{C}{l}, \qquad  \sup_{x \in \R^3} \| \hess^{\R^3} \chi (x) \| \le \frac{C}{l^2}.
\end{equation}
Now let us define the function $f_l \colon \R^3 \to \R$ as follows:
$$
f_l(x) \coloneqq f(x) + M \chi_l(x),
$$
where $M \coloneqq \sup f$.

Observe that $f_l$ is bounded. In fact,
\begin{equation}
R \le f_l \le 2 M.
\end{equation}
Moreover, observe that
\begin{equation}\label{eq_sup_f_l}   
\sup_{\widetilde{\Sigma}\cap \pi^{-1}(\mathcal{A}_{s^*})} f_l > R +  M = \sup_{\widetilde{\Sigma} \, \setminus \, \pi^{-1}(\mathcal{A}_{s^*})} f_l
\end{equation}
and also note that $f_l$ is smooth on $\widetilde{\Sigma} \, \cap \, \pi^{-1}(\mathcal{A}_{s^*})$ away from $\partial \widetilde{\Sigma}$. 
The Omori-Yau maximum principle yields the existence of  a sequence $(p_k) \subseteq\widetilde{\Sigma}  \cap  \pi^{-1}(\mathcal{A}_{s^*})$ satisfying the following properties:
\begin{enumerate}[(i)]
\item $\lim_{k \rightarrow \infty} f_l(p_k) = \sup_{\Sigma}f_l$,
\item $\lim_{k \rightarrow \infty} \nabla^\Sigma f_l(p_k) = 0$,
\item $\lim_{k \rightarrow \infty} \Delta^{\Sigma} f_l(p_k) \le 0$.
\end{enumerate}
Such a sequence $(p_k)$ is said to be an \emph{Omori-Yau sequence} for $f_l$.

We now distinguish two cases and we will see that they both lead to a contradiction. Let us assume first that there exists $l_0>0$ for which $f_{l_0}$ admits an Omori-Yau sequence $(p_k) \subseteq \widetilde{\Sigma} \cap \pi^{-1}(\mathcal{A}_{s^*}) $  with $x_3(p_k)$ unbounded in the $+\infty$ direction. 
Therefore for $k$ large enough, we have that $x_3(p_k) \ge l_0$ and thus $f_{l_0}(p_k) = f(p_k) + M$, 
 $\nabla^{\Sigma} f_{l_0}(p_k) = \nabla^{\Sigma} f(p_k)$ and  $\Delta^{\Sigma} f_{l_0}(p_k) = \Delta^{\Sigma} f(p_k)$. Therefore we have that
\begin{align}\label{equation_OY_nabla_f} 
\lim_{k \rightarrow \infty} \nabla^\Sigma f(p_k) = 0  
\end{align}
and 
\begin{equation}\label{equation_OY_lap_f}
\lim_{k \rightarrow \infty} \Delta^{\Sigma} f(p_k) \le 0.
\end{equation}

Note that on $\pi^{-1}(\mathcal{A}_{s^*})$, we have that $\nabla^{\R^3} f$ is a unit vector field, since it is the gradient of a distance function.  Observe that from \eqref{equation_OY_nabla_f}, and from the decomposition 
$$
\|\nabla^{\Sigma} f\|^2 = \Big\|\nabla^{\R^3} f\Big\|^2 - \Big\|\lef\nabla^{\R^3} f\rig^{\perp}\Big\|^2,
$$
we have that $ \lim_{k \rightarrow \infty} |\langle \nabla^{\R^3} f(p_k), \nu(p_k)\rangle  | = 1$.

Since $\nabla^{\R^3} f \perp e_3$, this implies 
\begin{equation}\label{equation_nu_OY}%
\lim_{k \rightarrow \infty} \langle \nu(p_k), e_3 \rangle = 0.
\end{equation}
From  \eqref{omori_yau_eq_f}, \eqref{equation_nu_OY} and \eqref{equation_OY_nabla_f}, we obtain
\begin{equation}
\lim_{k \rightarrow \infty} \Delta^{\Sigma} f(p_k) = \frac{1}{\lim_{k \rightarrow \infty} f(p_k)} =  \frac{1}{\sup_{\Sigma} f_{l_0} -M} > 0
\end{equation}
and this is in contradiction with  \eqref{equation_OY_lap_f}.

Let us now assume that for every $l > 0$, every Omori-Yau sequence has bounded $x_3$-coordinate. This implies, since $\Sigma$ is proper, that $f_l$ attains its maximum at some point $q_l \in \widetilde{\Sigma} \cap \pi^{-1}(\mathcal{A}_{s^*})$. Therefore we have
\begin{enumerate}[(i)]
\item $f_l(q_l) = \sup_{\Sigma} f_l$,
\item $\nabla^{\Sigma}f_l (q_l) = 0$,
\item $ \Delta^{\Sigma} f_l(q_l) \le 0$.
\end{enumerate} 

From the estimates \eqref{estimates_chi}, we can estimate the gradient of $f$ at $q_l$,
\begin{align*}
\| \nabla^{\Sigma} f (q_l) \| &= \| \nabla^{\Sigma}f_l(q_l) - M \nabla^{\Sigma} \chi_l(q_l) \| \\
&=  \|M \nabla^{\Sigma} \chi_l(q_l) \| \le \|M \nabla^{\R^3} \chi_l(q_l) \| \le \frac{C}{l}.
\end{align*}
Taking the limit for $l \rightarrow \infty$, we have
\begin{equation}\label{limit_gradient_pseudo_OY} 
\lim_{l \rightarrow \infty} \|\nabla^{\Sigma} f(q_l)\| = 0.
\end{equation}
Note that from \eqref{estimates_chi} we can estimate the Laplacian $\Delta^{\Sigma} \chi$ as follows:
\begin{align*}
|\Delta^{\Sigma} \chi_l| &= \left|\Delta^{\R^3} \chi_l - \hess^{\R^3}\!\chi_l(\nu, \nu) + H \langle \nu, \nabla^{\R^3} \chi_l \rangle \right| \\
&\le \frac{C}{l^2} + \frac{C}{l}.
\end{align*}
Therefore, we obtain
\begin{equation}\label{limsup_laplacian_pseudo_OY} 
\limsup_{l \rightarrow \infty} \Delta^{\Sigma} f(q_l) \le 0.
\end{equation}
On the other hand, if we evaluate \eqref{omori_yau_eq_f} at points $q_l$, by using \eqref{limit_gradient_pseudo_OY}, we have  
$$
\lim_{l \rightarrow \infty} \Delta^{\Sigma} f(q_l) > 0
$$
and this is in contradiction with \eqref{limsup_laplacian_pseudo_OY}.
This completes the proof of Step~(i).

\textbf{Step (ii)}: Let us now prove that $x_3$ is uniformly bounded from above on 
$$
\Sigma_\delta \coloneqq \Sigma \cap \{x \in \R^3 \colon 0 \le x_2 \le \delta \}
$$ 
for every $0 < \delta < \varrho^*$.  

Let us decompose $\Sigma_\delta$ as 
$$
\Sigma_\delta = \Sigma_+ \cup \Sigma_-,
$$
where
 $\Sigma_+ \coloneqq \Sigma_\delta \cap \{x \in \R^3 \colon x_1 \ge 0 \}$,  similarly $\Sigma_- \coloneqq \Sigma_\delta \cap \{x \in \R^3 \colon x_1 \le 0 \}$. 
We are going to show that $x_3$ is bounded from above separately on $\Sigma_+$ and $\Sigma_-$.

We prove the claim for $\Sigma_+$ only, since the considerations for $\Sigma_-$ are analogous.
 Let us consider the grim reaper cylinder 
$$
\Gamma \coloneqq \left\{ (x_1, x_2, x_3) \colon x_3 = -\log\left( \cos \left( x_1 + \frac{\pi}{2}\right)\right), \, -\pi < x_1 < 0 \right\}.
$$

Observe that $\Gamma \cap \Sigma_+ = \emptyset$.

 Let $F_\theta^\delta \colon \R^3 \to \R^3$ be the rotation of angle $\theta \in [0, \frac{\pi}{2})$ around the vertical line $\{(0, \delta)\} \times \R$.
Let us define the following 1-parameter family 
$$
\Gamma_\theta^\delta \coloneqq F_\theta^\delta(\Gamma).
$$
Note that 
\begin{equation}\label{boundary_sigma_p}
\partial \Sigma_+ = \{ x \in \Sigma \colon  x_1 = 0 \text{ and } 0 \le x_2 \le \delta \} 
 \cup   \{x \in \Sigma   \colon  x_2 = \delta \text{ and }  x_1 \ge 0\}.
\end{equation}

\begin{center}
\begin{figure}
\begin{tikzpicture}
\draw[->] (-3,0) -- (5.1,0);
\draw (5.1, 0) node [below] {$x_1$};
\draw[->] (0, -2) -- (0, 3.6);
\draw (0, 3.6) node [below right] {$x_2$};
\draw[dashed] (-3, 1) -- (5,1);
\draw (-3, 1) node [left] {$x_2 = \delta$};
\draw[dashed] (-3, 1.5) -- (5,1.5);
\draw (-3, 1.5) node [left] {$x_2 = \varrho^*$};
\draw [dashed, red] (-1.5, -1.5) -- (-1.5, 3.5);
\draw [dashed, red] (0, -1.5) -- (0, 3.5);
\fill [red, opacity = 0.2](-1.5, -1.5) rectangle (0,3.5 );
\coordinate (A) at (0,1);
\draw [dashed, red, rotate around = {30:(A)}] (-1.5, -1.5) -- (-1.5, 3.5);
\draw [dashed, red, rotate around = {30:(A)}] (0, -1.5) -- (0, 3.5);
\fill [red, opacity = 0.2, rotate around = {30:(A)}](-1.5, -1.5) rectangle (0,3.5 );
\draw[red, ->, thick] (0,-1) arc (-90: -60: 2) node[below left]{$\theta$};
\draw (-1, -1.5) node [below] {$\pi(\Gamma)$};
\draw (0.2, -1.4) node [below right] {$\pi(\Gamma_\theta^{\delta})$};
%
\fill[blue, opacity = 0.2] (0,1) rectangle (5,0);
\draw (3, 0.5) node[] {$\pi(\Sigma_+)$};
\draw[very thick] (0, 0) -- (0, 1);
\draw[very thick] (0,1) -- (5,1);
%
\fill (0,1) circle[radius=1pt] ;
\end{tikzpicture}
\caption{} \label{fig_2}
\end{figure}
\end{center}

Because of assumption \eqref{assumption_prop_new_refinement}, we have that for every $\theta \in [0, \frac{\pi}{2})$, \begin{equation}\label{first_piece_boundary}
\Gamma_\theta^\delta \cap  \Sigma  \cap  \{ (0, x_2, x_3) \in \R^3 \colon  0 \le x_2 \le \delta \}  = \emptyset.
\end{equation}
In fact, the grim reaper cylinder $\Gamma$ is the graph of a convex and nonnegative function, therefore the $x_3$-coordinate function is nonnegative on each $\Gamma^{\delta}_\theta$. Moreover from the construction of the family $\Gamma_\theta^\delta$, we have that for every $\theta \in [0, \frac{\pi}{2})$,
\begin{equation}\label{second_piece_boundary}
\Gamma_\theta^\delta \cap  \Sigma  \cap \{(x_1, \delta, x_3) \colon x_1 \ge 0\} = \emptyset.
\end{equation}
Therefore, combing \eqref{first_piece_boundary} and \eqref{second_piece_boundary} with \eqref{boundary_sigma_p},  we conclude that
\begin{equation}\label{no_inters_boundary}
\Gamma_\theta^\delta \cap \partial \Sigma_+ = \emptyset 
\end{equation}
for every $\theta \in [0, \frac{\pi}{2})$.

We want to prove that $\Gamma_\theta^\delta \cap \Sigma_+ = \emptyset$ for every $\theta \in [0, \frac{\pi}{2})$.   Recall that $\Gamma_0^\delta \cap \Sigma_+ = \Gamma \cap \Sigma_+ = \emptyset$.
Consider the function
$$
\theta \mapsto \dist(\Gamma^{\delta}_\theta, \Sigma_+) =  \dist(F^{\delta}_\theta(\Gamma), \Sigma_+).
$$
It is clearly a continuous and nonnegative on $[0, \frac{\pi}{2})$, since it is the composition of two continuous functions.
 We want to prove that it is actually strictly positive on $[0, \frac{\pi}{2})$. Assume for contradiction that this is not the case and let 
$$
\theta^* \coloneqq \min  \left\{\theta \in \left[0, \frac{\pi}{2}\right) \colon \dist(\Gamma^{\delta}_\theta, \Sigma_+) = 0\right\}.
$$
Observe  that $\pi(\Gamma^\delta_\theta) \cap \pi(\Sigma_+)$ is a triangle for each $\theta \in [0, \frac{\pi}{2})$ (see Figure \ref{fig_2}). From  Step (i), we have that the $x_3$-coordinate is bounded from above on $\pi^{-1}\lef \pi(\Gamma^\delta_\theta) \cap \pi(\Sigma_+) \rig \cap \Sigma_+$ and the $x_3$-coordinate is bounded from belove (is nonnegative) on $\Gamma^\delta_\theta$. Thus, since $\Sigma_+$ and $\Gamma_\theta^\delta$ are properly immersed,  the distance between $\Gamma^\delta_\theta$ and $\Sigma_+$ is always attained. In particular we have
$$
\dist(F^{\delta}_\theta(\Gamma), \Sigma_+) = 0 \Leftrightarrow F^{\delta}_\theta(\Gamma) \cap \Sigma_+ = \emptyset,
$$
thus there exists $p \in \Gamma_{\theta^*}^\delta \cap \Sigma_+$. 
From \eqref{no_inters_boundary}, we have that $p \in \lef \Sigma_+ \setminus \partial \Sigma_+ \rig$. But this is in contradiction with the separating tangency principle  (see Lemma 2.4 in \cite{mo14}).

Similarly, one can show that $\Gamma_{\theta}^\delta \cap \Sigma_- = \emptyset$ for $\theta \in (-\frac{\pi}{2}, 0]$. This implies that 
\begin{equation}\label{sigma_bounded_by_gr}
\Sigma_\delta \cap \Gamma_{\frac{\pi}{2}} = \emptyset.
\end{equation}
Note that 
$$
\Gamma_{\frac{\pi}{2}} = \Gamma_{-\frac{\pi}{2}} = \left\{ \lef x_1, x_2,-\log\left(\cos\left( x_2 - \delta + \frac{\pi}{2}\right)\right)\rig,   \delta -\pi < x_2 < \delta \right\}.
$$

In other words, the grim reaper cylinder $\Gamma_{\frac{\pi}{2}}$ lies ``above'' $\Sigma_\delta$. 
Observe that \eqref{sigma_bounded_by_gr} holds for every $0 < \delta < \varrho^*$ (note that the domain of $\Gamma_{\frac{\pi}{2}}$ depends on $\delta$). This finishes the proof of Step (ii).

\textbf{Step (iii)}:  We will now finally show that $\Sigma_\delta = \emptyset$ for every $\delta < \varrho^*$. 
Thanks to Step (ii), we can assume w.l.o.g. 
\begin{equation}\label{eq_hypotesis_step_3}
\sup_{\Sigma_\delta} x_3 < 0.
\end{equation}
Let $S \subseteq \R^3$ be a $\Delta$-wing translater (see \cite{blt18} and \cite{himw19a}) such that it is the graph of a convex function $u \colon \Omega \subset \R^2 \to \R$, where $\Omega$ is the strip
$$
\Omega \coloneqq \{(x_1, x_2) \in \R^2 \colon - \gamma < x_2 < \delta\}
$$
for some $\gamma > 0$ such that $\gamma + \delta > \pi$. Let us now define a one parameter family of translaters with boundary $\tilde{S}_t$ as follows:
$$
\widetilde{S}_t \coloneqq \lef S + t e_3 \rig \cap \{x \in \R^3\colon x_3 \le 0\}.
$$ 
Note that $\widetilde{S}_t$ is compact and $\partial (\tilde{S}_t ) = (S + t e_3) \cap \{x_3 = 0\}$. Observe that
\begin{equation}
\bigcup_{t \in \R} \widetilde{S}_t = \Omega \times (-\infty, 0].
\end{equation}
From the way we chose $\Omega$, we have that $\Sigma \cap (\Omega \times (-\infty, 0] )\ne \emptyset$.

Since $\widetilde{S}_t$ is compact for every $t \in \R$ and since  $\Sigma$ is properly immersed, there exists $t^* \in \R$ such that $\widetilde{S}_{t^*} \cap \Sigma \ne \emptyset$ and such that $\widetilde{S}_{t} \cap \Sigma = \emptyset$ for $t>t^*$. From \eqref{eq_hypotesis_step_3}, we have that any intersection point $p \in \widetilde{S}_{t^*} \cap \Sigma $ is an interior point for $\widetilde{S}_{t^*}$. We can therefore apply the separating tangency principle and get that $\Sigma = S + t^* e_3$, which is a contradiction because we assumed $\pi(\Sigma) \subseteq \{(x_1, x_2) \in \R^2 \colon x_2 \ge 0\}$.
\end{proof}



\begin{corollary}\label{corollary_of_proposition_refinement}
Let $\Sigma^2 \subseteq \R^3$ be a properly immersed translater contained in a slab. W.l.o.g. \!let us assume 
$$
\conv (\pi(\Sigma)) = \{ (x_1, x_2) \in \R^2 \colon |x_2| < \delta \},
$$
for some $\delta > 0$. Thus $\{x \in \R^3 \colon |x_2| = \delta\} = \pi^{-1}(\partial \conv(\pi(\Sigma)))$. Let $P$ be a vertical plane such that $P \nparallel \{x \in \R^3 \colon |x_2| = \delta\}$.

Then there exist two distinct sequences $(p^1_k), (p^2_k) \subseteq \Sigma \cap P$ satisfying the following properties:
\begin{enumerate}[(i)]
\item $\lim_{k \rightarrow \infty}x_3(p^1_k)= \lim_{k \rightarrow \infty} x_3(p^2_k)  =  \infty $, 
\item $\lim_{k \rightarrow \infty} \dist(p^1_k, L_1)  = \lim_{k \rightarrow \infty} \dist(p^2_k, L_2) = 0 $,
\end{enumerate} 
where $L_1$ and $L_2$ are the two vertical lines  $L_1 = \{x \in \R^3 \colon x_2 = \delta\} \cap P$ and $L_2 = \{x \in \R^3 \colon x_2 = -\delta\} \cap P$.
\end{corollary}

\begin{proof}
Assume by contradiction that the statement is not true. For instance, let us assume that there is no sequence $(p^1_k)$ satisfying $(i)$ and $(ii)$.  Then this means that $x_3 $ is bounded from above on 
$\{x \in \Sigma \cap P \colon \dist(x, L_1) \le \varepsilon\}$ for some $\varepsilon > 0$.
W.l.o.g. we can assume that 
\begin{equation}\label{x_3_neg_corollary}
x_3 < 0
\end{equation}
for every $x = (x_1, x_2, x_3) \in \Sigma \cap P$ such that $\dist(x, L_1) < \varepsilon$.

Let $\mathcal{H}$ be one of the two halfspaces such that $\partial \mathcal{H} = P$. Note that from Theorem \ref{new_refinement}, we can assume that $\mathcal{H} \cap \Sigma$ contains a sequence of points $(q_k) \subseteq \mathcal{H} \cap \Sigma$ such that $x_3(q_k) \nearrow \infty$ and $\dist(q_k, L_1) \rightarrow 0$. Let $\mathcal{C}$ be a vertical cylinder such that $\mathcal{C} \subseteq \mathcal{H} \cap \{x \in \R^3 \colon x_2 \le \delta\}$ and such that $\mathcal{C}$ is tangent to $P$ and to $\{x \in \R^3 \colon x_2 = \delta\}$.  Observe that $\pi \lef \lef \mathcal{H} \cap \{x \in \R^3 \colon x_2 \le \delta \} \rig \setminus \mathcal{C}\rig$ consists of two connected components, one bounded and another one unbounded. Let $\mathcal{A}$ be the bounded one. Moreover let $\mathcal{L}$ be the axis of the cylinder. Then we define the function $f \colon \{(x_1, x_2) \in \R^2 \colon |x_2| < \delta\} \to \R$ as follows
$$
f(x) \coloneqq
\begin{cases}
\dist(x, \mathcal{L}) \qquad &\text{ if } \pi(x) \in \mathcal{A}\\
R &\text{ if } \pi(x) \notin \mathcal{A}.
\end{cases}
$$
where $R$ is the radius of $\mathcal{C}$. Let us consider the restriction $f|_{\widetilde{\Sigma}}$,  where $\widetilde{\Sigma}$ is the translater with boundary defined as 
$$
\widetilde{\Sigma} \coloneqq \Sigma \cap \{x \in \R^3 \colon x_3 \ge 0\}.
$$
Note that, because of the existence of the sequence $(q_k)$, we have that
$$
\sup_{\widetilde{\Sigma}}(f) = \dist(\mathcal{L}, L) > R
$$
and from \eqref{x_3_neg_corollary} follows that $f|_{\widetilde{\Sigma}}$ is smooth on the set 
$$
\{x \in \widetilde{\Sigma} \colon f|_{\widetilde{\Sigma}}(x) > \sup_{\widetilde{\Sigma}} f - \varepsilon \}.
$$
We can therefore apply the Omori-Yau principle directly, without the need of the ``adiabatic trick'', in order to get a contradiction. The computations are similar (and simpler, since we do not need the cut-off function here) to the ones in the proof of Theorem \ref{new_refinement}.
\end{proof}

We include here a Bernstein type theorem for $1$-periodic translaters, which will not be used in the proof of Theorem \ref{theorem_slab} but it is worth mentioning. It is a simple consequence of Theorem \ref{new_refinement}.
\begin{corollary}[Bernstein type theorem for 1-periodic translaters]\label{bernstein_periodic}
Let $\Sigma^2 \subseteq \R^{3}$ be a properly immersed  translater such that $\Sigma \subset \mathcal{H}$, where $\mathcal{H}$ is a vertical halfspace. Let us assume that $\Sigma$ is $1$-periodic in the $e_{3}$-direction, i.e. there exists $a > 0$ such that 
\begin{equation}\label{1-period}
\Sigma = \Sigma + a e_{3}.
\end{equation}
Then $\Sigma$ is a vertical hyperplane.
\end{corollary}

\begin{proof}
Let us assume $\partial \mathcal{H} \subseteq \pi^{-1}\lef \partial \conv\lef \pi \lef \Sigma \rig \rig \rig$. From Theorem \ref{new_refinement} and the $1$-periodicity assumption \eqref{1-period}, it follows that $\Sigma \cap \partial \mathcal{H}$. The conclusion follows from the separating tangency principle. 
\end{proof}

\begin{remark}
Observe that nontrivial periodic translaters do exist but the known examples are in line with Corollary \ref{bernstein_periodic}, because their period is a vector orthogonal to $e_{3}$ (see \cite{ng09} and the very recent paper \cite{hmw19b}).
\end{remark}

%
%
%
%
%
%
%


\section{The structure of the set $Z$.}\label{section_Z}

In this section we  assume $\Sigma$ to be a properly embedded translating soliton. We want to study the structure of the intersection of $\Sigma$ with a vertical plane $P$ and we denote such intersection as 
$$
Z\coloneqq \Sigma \cap P.
$$
Note that $Z$ can be described as the zero set of a function defined on $\Sigma$ as follows. Let $p \in Z$ and let $V \in \mathbb{S}^2$ be a unit vector orthogonal to $P$. Then $Z$ is the zero set of the function
$$
x \mapsto \langle V, x - p \rangle, \qquad x \in \Sigma.
$$

The structure of $Z$ is described by the following lemma, which is inspired by Lemma 6 in \cite{br16} and \cite{ro95}.

\begin{lemma}\label{structure_Z} Let us assume  that $\Sigma$ is not flat, i.e. is not a vertical plane. Then for each point $x \in Z$ there exists an open neighborhood $x \in U \subseteq \Sigma$, such that $Z \cap U$ is a union of finitely many $C^2$-arcs $\Gamma_1, \dots, \Gamma_m$ which intersect transversally at $x$. The number $m$ is the vanishing order of the function 
$
x \mapsto \langle V, x - p\rangle
$
at $p$.

Globally, the set $Z$ is the union of countably many 1-dimensional properly immersed $C^2$-submanifolds  without boundary of $\R^3$ and they may intersect pairwise only at isolated points.
\end{lemma}

\begin{proof}
Let $f(x) \coloneqq \langle V, x - p \rangle$. Observe that $\nabla^{\Sigma}f = V^\top$.  Moreover, using the translater equation \eqref{translating_equation} and the fact that $V\perp e_3$, we have
\begin{align*}
\Delta^{\Sigma} f &= \di^\Sigma(V^\top)  =\di^{\Sigma}(V) - \di^{\Sigma}(V^{\perp})\\
&= - \langle V, \nu \rangle H = \langle V, e_3^{\perp} \rangle = - \langle V, e_3^\top \rangle = - \langle \nabla^\Sigma f, e_3 \rangle.
\end{align*} 
Thus  $f$ satisfies the following elliptic equation
\begin{equation}\label{pde_f_Z}
\Delta^\Sigma f + \langle \nabla^{\Sigma} f, e_3 \rangle = 0.
\end{equation}
Therefore,
\begin{align*}
\Delta^{\Sigma} (e^{\frac{x_3}{2}} f) &= \di^\Sigma( \nabla^{\Sigma} (e^{\frac{x_3}{2}} f) ) \\
&= \di^\Sigma \lef e^{\frac{x_3}{2}} \frac{f}{2} e_3^T +  e^{\frac{x_3}{2}} \nabla^\Sigma f \rig \\
&= e^{\frac{x_3}{2}} \lef \frac{f}{4} |e^T_3|^2 + \frac{f}{2} \di^\Sigma(e^\top_3) + \langle \nabla^\Sigma f, e_3 \rangle + \Delta^{\Sigma} f \rig \\
&=\lef e^{\frac{x_3}{2}}f \rig \lef \frac{|e^T_3|^2}{4} + \frac{\di^\Sigma(e^\top_3)}{2} \rig. 
\end{align*}
The conclusion of the first part of the statement follows from applying Theorem 2.5 in \cite{ch76} to the function $x \mapsto e^{\frac{x_3}{2}}f(x)$ and observing that its zero set coincides with the zero set of $f$.

The second part of the statement follows immediately from the first part and from the properness of $\Sigma$. 
\end{proof}

\begin{remark}\label{remark_Z}
We are mainly interested in the special case where  
$$
H(p) = 0,\quad  P = T_p \Sigma, \quad V = \nu(p),
$$
where $T_p \Sigma$ denotes, with a little abuse of notation, the geometric tangent plane of $\Sigma$ at $p$. Observe that from equation \eqref{translating_equation}, $H(p) = 0 $ if and only if $T_p\Sigma$ is a vertical plane. 
Note that in this case $f \colon x \mapsto \langle V, x -p  \rangle$ has vanishing order $m \ge 2$ at $p$, because $\nabla^\Sigma f = V^\top$ and $V = \nu(p)$, we have that   $\nabla^\Sigma f|_{p} = 0$. Therefore there exists a neighborhood $U$ of $p$ such that $Z \cap U$ consists of at least two $C^2$-curves intersecting transversally at $p$. 
\end{remark}
We have also the following  information about $Z$. 

\begin{lemma}\label{refinement_Z}
Under the same assumptions as Lemma \ref{structure_Z}, if we further assume $\Sigma$ to be simply connected, then each connected component of $Z$ is simply connected. In particular,  $Z$ is the union of the images of countably many,  $C^2$-embeddings $\gamma_j \colon \R \to \Sigma$ which may intersect pairwise at most at one point. 
\end{lemma}

\begin{proof}
Assume by contradiction that there exists a continuous and injective loop $\delta \colon \mathbb{S}^1 \to Z$ which is not homotopically trivial. Then  $\delta$ is a Jordan curve in $\Sigma$, and since we are assuming $\Sigma$ to be homeomorphic to the plane, from the Jordan theorem, the image of $\delta$ is the boundary of a nonempty, bounded open set $\Omega \subseteq \Sigma$. This means that the function $f(x) = \langle V, x -p  \rangle$ satisfies the following boundary problem:
\begin{align*}
\begin{cases}
\Delta^{\Sigma} f + \langle \nabla^\Sigma f, e_3 \rangle = 0 \qquad &\text{in } \Omega  \\
f = 0 \qquad &\text{on } \partial \Omega.
\end{cases}
\end{align*}
From the maximum principle, it follows that $f$ is identically zero in $\Omega$, which means $\Omega \subseteq Z$. But this contradicts Lemma \ref{structure_Z}. 
\end{proof}


\begin{lemma} \label{corollary_topology}
Let $\Sigma^2 \subseteq \R^3$ be a simply connected, properly embedded translater.
Let $\mathcal{H} \subseteq \R^3 $ be a vertical halfspace. 

Then each connected component of  $\Sigma \cap \mathcal{H}$  is simply connected.
\end{lemma}

\begin{proof}
Let $P$ be the vertical plane $P = \partial \mathcal{H}$, $p \in Z = P \cap \Sigma$ and let $V$ be the orthogonal unit vector to $P$ pointing outside $\mathcal{H}$. Let us assume, for contradiction, that there exists an embedding  $\gamma \colon \mathbb{S}^1 \to \Sigma \cap \mathcal{H}$ which is not homotopically trivial in $\Sigma \cap \mathcal{H}$.        Since $\Sigma$ is simply connected, there exists $\Omega \subseteq \Sigma$ such that 
$ \partial \Omega = \gamma(\mathbb{S}^1) $. 
Let $f \colon \Sigma \to \R$ be defined as $f(x) = \langle V, x - p \rangle$ as above. Observe that 
\begin{equation*}
f|_{\partial \Omega} \le 0.
\end{equation*} 
Since we are assuming $\gamma$ is not homotopically trivial in $\Sigma \cap \mathcal{H}$, we have that $\Omega \nsubseteq \Sigma \cap \mathcal{H}$. This implies
\begin{equation}\label{max_f_violation_MP}
 \max_{\Omega} f > 0 \ge \max_{\partial \Omega} f.
\end{equation}
On the other hand  $f$ satisfies the elliptic equation \eqref{pde_f_Z}, therefore \eqref{max_f_violation_MP} violates the maximum principle.

\end{proof}

\section{The structure of $\{H=0\}$}\label{section_H}

In this section we study the zero set of the mean curvature of $\Sigma$. 

\begin{remark}\label{remark_H_manifold}
On a translater $\Sigma$, the mean curvature  $H$ solves the following equation:
\begin{equation}\label{pde_H}
\Delta^\Sigma H  + \langle \nabla^\Sigma H, e_3 \rangle + |A|^2 H = 0, 
\end{equation}
see for instance Lemma 2.1 in \cite{mss15}.
As in the proof of Lemma \ref{structure_Z}, one can readily check that $e^{\frac{x_3}{2}}H$ satisfies the equation (without first order term):
$$
\Delta^\Sigma (e^{\frac{x_3}{2}} H ) = \lef e^{\frac{x_3}{2}} H \rig h,
$$
for some smooth function $h$. Observe that the zero set of $e^{\frac{x_3}{2}} H$ coincides with $\{H = 0\}$. If $\Sigma$ is not flat, from Theorem 2.5 in \cite{ch76}, we have that it is a union of $1$-dimensional $C^2$-manifolds and the singular points (namely the intersection points of such $1$-dimensional manifolds) are isolated.

\end{remark}

\begin{lemma}\label{lemma_no_planar}
Let $\Sigma^2 \subseteq \R^3$ be a complete translater, such that the unit normal vector field $\nu$ is constant along $\{H = 0\} $. 

Then $\Sigma$ is mean convex.
\end{lemma}

\begin{proof}
We can assume $\{H = 0\} \ne \emptyset$ and that $\Sigma$ is not flat, otherwise the statement is trivially true. 
Let us assume that $\nu$ is constant along $\{H = 0\}$. Let $V \in \mathbb{S}^2$ be such that $\nu|_{\{H =0\}} \equiv V$. Note that from \eqref{translating_equation} we have that $V \perp e_3$.  From Remark \ref{remark_H_manifold} we have that $\{ H = 0\}$ is a $1$-dimensional smooth manifold away from  a 
 set of isolated points. 

 Let $p \in \{H = 0\}$ be a regular point and let $\gamma \colon (-\varepsilon, \varepsilon) \to \{H = 0\}$ be a regular curve such that $\gamma(0) = p$. Since $\nu$ is constant along $\{H = 0\}$, we have that $T_{\gamma(s)} \Sigma = T_p \Sigma$ and  $\gamma(s) \in T_p \Sigma \cap \Sigma = Z$, for every $s \in (-\varepsilon, \varepsilon)$. 

From Lemma \ref{structure_Z} we have that there exists a neighborhood $p \in U \subseteq \Sigma$ such that $Z \cap U$ is the union of finitely many $C^2$-arcs intersecting transversally at $p$. Moreover, we can assume that $p$ is the only singular point of the $1$-dimensional $C^2$-manifold $Z \cap U$.

Observe that the function $x \mapsto \langle V, x- p \rangle$ has vanishing order $m \ge 2$ at $\gamma(s)$ for every $s \in (-\varepsilon, \varepsilon)$. Therefore, from Lemma \ref{structure_Z}, we have that each point $\gamma(s)$ is a singular point of $\{H = 0\}$ and this is in contradiction with the fact that $p$ is an isolated singular point.
\end{proof}


We conclude this section with the following proposition which  will not be used in the proof of Theorem \ref{theorem_slab}, but is a stand-alone observation.

\begin{proposition}\label{lemma_compactness}
Let $\Sigma^2 \subseteq \R^{3}$ be a   complete translater with only one  end and  assume  $\{H = 0\}$ to be compact.

Then $\{H = 0\}$  is empty. Namely, $\Sigma$ is strictly mean convex.
\end{proposition}

\begin{proof}
Let us assume by contradiction that $\{ H = 0\}$ is compact and non-empty.
From Remark \ref{remark_H_manifold}, we have that it is a $1$-dimensional smooth manifold away from a closed set of isolated points. Therefore, since we are assuming $\{H = 0\}$ to be compact, the singular set is a union of finitely many points.

 Since $\Sigma$ has one end, we have that either $\{H \ge 0\}$ or $\{H \le 0\}$ is compact. Let us assume without loss of generality that $\Omega \coloneqq \{H \ge 0\}$ is compact. Since $H$ solves the elliptic equation \eqref{pde_H}, as an application of the strong maximum principle  applied to $H$,  we have that the interior $\{H > 0\}$ is non-empty, unless $\Sigma$ is flat.  Observe that $\Omega$ is a compact translater with boundary $\partial \Omega = \{H = 0\}$.

Let $V \in \R^3$ be a vector such that $\langle V, e_3\rangle = 0$. Let $P_V \coloneqq \{ x \in \R^3 \colon \langle V, x \rangle = 0\}$ and let us consider the one parameter family of planes $P_{V, t} \coloneqq P_V + t V$, with $t \in \R$. Since $\Omega$ is compact, there exists $t^* =  t^*(V)$ such that $P_{V, t} \cap \Omega = \emptyset$ for every $t < t^*$ and $P_{V,t^*} \cap \Omega \ne \emptyset$. Let $p \in P_{V, t^*} \cap \Omega$. Observe that $P_{V, t^*}$ is also a translater. Therefore, if $p \in \Omega \setminus \partial \Omega$, we get a contradiction from the separating tangency principle for translaters (Lemma 2.4 in \cite{mo14}). Thus $P_{V, t^*} \cap \Omega \subseteq \partial \Omega$. 

Since $\partial \Omega$ has at most finitely many singular points, we can choose $V$, such that there exists $p \in P_{V, t^*} \cap \Omega \subseteq \partial \Omega$ which is not a singular point.
 From the translater equation \eqref{translating_equation}, we have that the geometric tangent space $T_p\Sigma$ coincides with $P_{V, t^*}$. Since $\partial \Omega$ is regular at $p$, we get a contradiction also in this case from the boundary version of the separating tangency principle (see for instance Theorem 2.1.1 in \cite{pe16}). 
\end{proof}


\begin{remark} 
Observe that if $\Sigma$ has more than one end, then $\{H = 0\}$ can be non-empty and compact. Consider for example the wing-like translater introduced in \cite{css07}.
\end{remark}

\section{Proof of Theorem \ref{theorem_slab}}\label{section_proof}

\begin{proof}
The proof proceeds by contradiction. Let $\Sigma$ be as in the assumptions of Theorem \ref{theorem_slab} and let us assume for contradiction that $\Sigma$ is not mean convex.

Since $\Sigma$ has finite entropy and $|H| \le 1$, Lemma \ref{lemma_finite_entropy_properness} in the Appendix  implies that $\Sigma$ is properly embedded. Therefore, from the results in \cite{cm19} (see Remark \ref{remark_asymptotic_planes}) we have that $\conv(\pi(\Sigma))$ is a strip. Let $\mathcal{S}$ be the slab $\mathcal{S} \coloneqq \pi^{-1}\lef \conv (\pi(\Sigma)) \rig$.

From Lemma \ref{lemma_no_planar}, we can find a point $p \in \{H = 0\}$, such that $T_p \Sigma $ is not parallel to $\partial \mathcal{S}$. Note that $T_p\Sigma$ is a vertical plane, because of \eqref{translating_equation}. Observe that $\mathcal{S} \cap T_p \Sigma$ is a vertical strip on which $x_1$ and $x_2$ are bounded and $x_3$ is unbounded. 
From Lemma~\ref{structure_Z} and Lemma~\ref{refinement_Z}, the set $Z = \Sigma \cap T_p\Sigma$ is the union  of the images of countably many (possibly finitely many) 
 $C^2$-embeddings $\gamma_j \colon \R \to \Sigma$. Each of these $1$-dimensional submanifolds  is properly embedded in $\R^3$ and since the coordinates $x_1$ and $x_2$ are bounded on $Z$, we have  that for each $j$ the two limits $\lim_{t \rightarrow +\infty} x_3( \gamma_j(t))$ and $\lim_{t \rightarrow -\infty} x_3( \gamma_j(t))$ exist and each of them is equal to $+\infty$ or $-\infty$.

In what follows, we use the term ``ray'' to denote a half curve, i.e. to denote $\gamma_j^+ \coloneqq \gamma_j|_{[0, \infty)}$ or  $\gamma_j^- \coloneqq \gamma_j|_{(- \infty, 0]}$.


\textbf{Case 1}: Let us assume that there are at least $3$  rays in $Z$ for which their  $x_3$ coordinates goes to $+ \infty$. We will find a contradiction with the bound on the entropy.  This implies that there are at least three distinct sequences of points $(q^1_k), (q^2_k), (q^3_k) \subseteq Z$ such that 
$$
x_3(q^1_k) = x_3(q^2_k) = x_3(q^3_k) = k
$$
for every sufficiently large $k \in \N$. From Corollary \ref{corollary_of_proposition_refinement}, we can assume
\begin{equation}\label{eq_q_1_2}
\dist(q^1_k, L_1 ) \xrightarrow[k \rightarrow 0]{} 0, \qquad \dist(q^2_k, L_2 ) \xrightarrow[k \rightarrow 0]{} 0,
\end{equation}
where $L_1$ and $L_2$ are the two vertical parallel lines such that $L_1 \cup L_2 = P \cap \partial \mathcal{S}$.  Moreover, since $\pi(P \cap \mathcal{S})$ is compact, we can assume, up to extracting a subsequence, that 
\begin{equation}\label{eq_q_3}
\pi(q^3_k) \xrightarrow[k \rightarrow \infty]{} q
\end{equation}
for some $q \in \pi(P \cap \mathcal{S})$.
 Let us consider the sequence of translaters $(\Sigma_k)$, defined as
$$
\Sigma_k \coloneqq \Sigma - k e_3.
$$
Let us define the sequences $(\tilde{q}_k^i) \subseteq \Sigma_k$, for $i = 1, 2, 3$ as follows:
$$
\tilde{q}^i_k \coloneqq q^i_k - k e_3.
$$
From Proposition \ref{curvature_estimate}, we know that the norm of the second fundamental form of $\Sigma_k$ is uniformly bounded by a constant. Moreover, from \eqref{eq_q_1_2} and from \eqref{eq_q_3}, we have that
$$
\tilde{q}_k^1 \xrightarrow[k \rightarrow \infty]{} \pi(L_1), \qquad \tilde{q}_k^2 \xrightarrow[k \rightarrow \infty]{} \pi(L_2), \qquad \tilde{q}_k^3 \xrightarrow[k \rightarrow \infty]{} q.
$$

 Therefore, by employing a standard Arzelà-Ascoli argument (see for instance Theorem 2.14 in \cite{bgm19}), we have that there exists a properly embedded, not necessarily connected, smooth translater $\Sigma_\infty$, such that, up to a subsequence, we have
$$
\Sigma_k \xrightarrow[k \rightarrow \infty]{C^\infty_{loc}} \Sigma_\infty.
$$
Moreover, we have that $\Sigma_\infty \subseteq \mathcal{S}$, $L_1 \cap \Sigma_\infty \ne \emptyset$ and $L_2 \cap \Sigma_{\infty} \ne \emptyset$. Therefore, from the separating tangency principle, we can conclude that  $\Sigma_{\infty}$ is the following disjoint union
$$
\Sigma_{\infty} = P_1 \cup P_2 \cup \Sigma^{'},
$$
where $P_1$ and $P_2$ are the two vertical parallel planes such that $P_1 \cup P_2 = \partial \mathcal{S}$ and $\Sigma^{'}$ is a complete translater passing through $q$.
Corollary \ref{corollary_additivity} and Remark~\ref{rmk_entropy_1} in the Appendix implies that
$$
\lambda(\Sigma_\infty) = \lambda(P_1) + \lambda(P_2) + \lambda(\Sigma') \ge 3.
$$

Observe that $q$ might coincide with $\pi(L_1)$ or $\pi(L_2)$ and in that case $\Sigma^{'}$ would coincide with $P_1$ or $P_2$, respectively. This is not a problem because  in this situation, the convergence of $\Sigma_k$ to $P_1$ or $P_2$ would be of multiplicity at least $2$. 

Let $\mathcal{B}_R$ denote the ball in $\R^3$ of radius $R> 0$ centered at $0$. Observe that, for any $x_0 \in \R^3 $ and $t_0> 0$ we have
\begin{align}\label{align_f_functional_entropy}
\lambda(\Sigma) & = \lambda(\Sigma_k) \ge F_{x_0, t_0} (\Sigma_k) \ge F_{x_0, t_0}\lef\Sigma_k \cap \mathcal{B}_R \rig.
\end{align} 
The first equality in \eqref{align_f_functional_entropy} follows from the translation invariance of the entropy. 
Taking the limit for $k \rightarrow \infty$ in \eqref{align_f_functional_entropy} and using the fact that $\lim_{k \rightarrow \infty}F_{x_0, t_0} \lef \Sigma_k \cap \mathcal{B}_R\rig = F_{x_0, t_0} \lef \Sigma_{\infty}\cap \mathcal{B}_R\rig$, we obtain
\begin{equation}\label{entropy_R}
\lambda(\Sigma) \ge F_{x_0, t_0} (\Sigma_\infty \cap \mathcal{B}_R).
\end{equation}
Inequality \eqref{entropy_R} holds for every $R>0$,  thus $\lambda(\Sigma)  \ge F_{x_0, t_0} (\Sigma_\infty)$. After taking  the supremum over $x_0 \in \R^3$ and $t_0> 0$, we finally obtain the following contradiction
$$
3 > \lambda(\Sigma) \ge \lambda(\Sigma_\infty) \ge 3.
$$

\textbf{Case 2}: Let us now assume that there are at most $2$ rays such that their $x_3$ coordinate goes to $+ \infty$. From Corollary \ref{corollary_of_proposition_refinement}, we know that $x_3$ can not be bounded from above on $Z$. Therefore there is at least one ray in $Z$ on which $x_3$ goes to $+\infty$. 

In what  follows $\mathcal{H}^+$ and $\mathcal{H}^-$ are the two open halfspaces with boundary $T_p\Sigma$, namely  
\begin{align*}
\mathcal{H}^+ \coloneqq \{ x \in \R^3 \colon \langle x - p, \nu(p) \rangle > 0\}, \\ \mathcal{H}^- \coloneqq \{ x \in \R^3 \colon \langle x - p, \nu(p) \rangle < 0\}.
\end{align*}
Moreover, $\overline{\mathcal{H}}^+$ and $\overline{\mathcal{H}}^-$ will denote the closure of $\mathcal{H}^+$ and $\mathcal{H}^-$ respectively. 

Let $U$ be a neighborhood of $p$ in $\Sigma$  as in Lemma \ref{structure_Z}. Therefore, 
$$
Z \cap U =\bigcup_{j=1}^m \Gamma_j,
$$ 
where $\Gamma_j$ are $C^2$-arcs meeting transversally at $p$ and $m \ge 2$ (see Remark \ref{remark_Z}). We can choose $U$ such that  each $\Gamma_j$ divides $U$ into two connected components. 
From Lemma \ref{corollary_topology}, the arcs $\Gamma_j$ intersect pair-wise only at $p$. 

Moreover, we can assume $U$ to be the graph of a function $u \colon B   \to \R$ for some ball $B \subseteq T_p \Sigma$. From the discussion above and from the separating tangency principle, $U  \setminus  Z$ is the union of $2m$ connected components $U^+_1, \dots,U^+_m, U^-_1, \dots, U^-_m $, where $U^+_j \subseteq \mathcal{H}^+$ and $U^-_j \subseteq \mathcal{H}^-$. 

We denote by ${\Omega}^+_j$   the connected component of $\Sigma \cap \mathcal{H}^+$ containing $U^+_j$ and similarly, we denote by $\Omega^-_j$ the connected component of $\Sigma \cap \mathcal{H}^-$ containing $U^-_j$.

Observe that from Lemma \ref{refinement_Z} and Lemma \ref{corollary_topology}, it follows that if $j \ne k$,  $U^\pm_j$ and $U^\pm_k$ belong to two distinct connected components of $\Sigma \cap \mathcal{H}^\pm$. In other words $\Omega^+_1, \dots, \Omega^+_m, \Omega^-_1, \dots, \Omega^-_m$ are all distinct. 
Moreover observe that from Lemma \ref{refinement_Z} we have that 
\begin{equation}\label{eq_boundary_omega}
\partial \Omega^+_j \cap \partial \Omega^+_k = \{p\}, \qquad \partial \Omega^-_j \cap \partial \Omega^-_k = \{p\},
\end{equation}
 for $j \ne k$.
Moreover, from Corollary \ref{corollary_of_proposition_refinement}, we have 
\begin{equation}\label{x_3_boundary_omega}
\sup_{\partial \Omega_j^+} x_3 = +\infty, \qquad \sup_{\partial \Omega_j^-} x_3 = +\infty,
\end{equation}
for $j = 1, \dots, m$.

Let $\widetilde{Z}$ be the connected component of $Z$ containing $p$. We will now  distinguish the following subcases.
\begin{enumerate}[\bfseries (a)]
\item The $x_3$ coordinate is bounded from above on $\widetilde{Z}$.  
\item $\widetilde{Z}$ contains one ray such that the $x_3$ coordinate goes to $+ \infty$.
\item $\widetilde{Z}$ contains two rays such that the $x_3$ coordinate goes to $+ \infty$.
\end{enumerate}

\textbf{(a)} Let us assume the coordinate $x_3$ to be bounded from above on $\widetilde{Z}$.  Since we are in \textbf{Case 2}, there can be at most 2 connected components of $Z$ for which $x_3$ goes to $+\infty$.  Note that \eqref{eq_boundary_omega} and \eqref{x_3_boundary_omega}, together with the fact that $m \ge 2$, imply that, in fact, $m = 2$ and  there are exactly 2 distinct connected components $Z_1$ and $Z_2$ of $Z$  on which $x_3$ goes to $+\infty$ and such that 
\begin{equation}\label{christmas_1}
Z_1 \subseteq \partial {\Omega}^+_1  \quad \text{and} \qquad Z_1 \subseteq \partial {\Omega}^-_1
\end{equation}
and 
\begin{equation}\label{christmas_2}
Z_2 \subseteq \partial {\Omega}^+_2 \quad \text{and} \qquad Z_2 \subseteq \partial {\Omega}^-_2.
\end{equation}
But this is in contradiction with the fact that $\Sigma$ is simply connected. Indeed, we can construct a loop in $\Sigma$ with base point $p$ which is not homotopically trivial as follows: let $\delta_1 \colon [0, l_1] \to \Sigma$ be a regular curve such that $\delta_1(0) = p$ and $\delta_1(l_1) \in Z_1$ and $\delta_1(t) \in  {\Omega}_1^+$ for $0 < t <l_2$. Let $\delta_2 \colon [0, l_2] \to \Sigma$ be another regular curve connecting $Z_1$ and $\{p\}$, such that  $\delta_2(0) = \delta_1(l_1)$, $\delta_2(l_2) = p$ and such that $\delta_2(t) \in{\Omega}_1^-$. Let $\delta = \delta_1 * \delta_2$ be the concatenation of $\delta_1$ and $\delta_2$. Observe that the existence of $\delta_1$ and $\delta_2$ is guaranteed by \eqref{christmas_1}. It is immediate to see that $\delta$ is not homotopically trivial, because $Z_1$ and $\widetilde{Z}$ are two distinct connected components of $Z$.

\textbf{(b)} Let us assume that $\widetilde{Z}$ contains one ray, such that the $x_3$ coordinate goes to $+ \infty$. One can find again a contradiction with a similar argument as in the subcase (a).

\textbf{(c)} Let us assume $\widetilde{Z}$ contains two rays $r_1$ and $r_2$, such that the $x_3$-coordinate goes to $+ \infty$. Since we are in \textbf{Case 2}, this implies that $x_3$ is bounded on all the other connected components of $Z$. 
For the sake of clarity, let us assume that both  rays are emanating from $p$ (it is easy to deal with the general case). Namely, let us assume that $r_i \colon [0, \infty) \to Z$ and $r_i(0) = p$ for $i = 1, 2$. Note that there cannot be any other ray "between" them, otherwise its $x_3$-coordinate would have to go to $+\infty$, violating the hypothesis of  subcase (c). Therefore, one of the connected components $U^+_1, \dots, U^+_m, U^-_1, \dots, U^-_m$ must have $r_1 \cap U$ and $r_2 \cap U$ as boundary in $U$. W.l.o.g., \!\!  let us assume $\partial U^+_1 = (r_1 \cup r_2) \cap U$. Observe that \eqref{eq_boundary_omega} implies $\partial {\Omega}_j^+ \cap (r_1 \cup r_2) = \{p\}$ for every $j = 2, \dots, m$. Therefore, $x_3$ is bounded from above on $\partial {\Omega}_j^+$ for $j = 2, \dots, m$ but this contradicts \eqref{x_3_boundary_omega}.

%
%

\end{proof}

\appendix
\section{Colding-Minicozzi's entropy}\label{appendix}


Let  $\Sigma^n \subseteq \R^{n+k}$ be a submanifold. Following \cite{cm12}, given $x_0 \in \R^{n+k}$ and $t_0 > 0$, the functional $F_{x_0, t_0}$ is defined as follows
\begin{equation}\label{F-functional}
F_{x_0, t_0} (\Sigma) \coloneqq \frac{1}{\lef 4 \pi t_0 \rig^{\frac{n}{2}}} \int_{\Sigma} e^{-\frac{\|x- x_0\|^2}{4 t_0}} \, d\mu(x). 
\end{equation}

Then the \emph{entropy} functional $\lambda(\Sigma)$ is defined as follows (see also \cite{mm09}):
\begin{equation}\label{entropy}
\lambda(\Sigma) \coloneqq \sup_{x_0 \in \R^{n+k}, \, t_0 > 0} F_{x_0, t_0}(\Sigma).
\end{equation}
The functionals $F_{(x_0, t_0)}$ and the entropy functional, naturally extend to Radon measures. 

\begin{remark}\label{rmk_entropy_1}
Observe that for any $n$-dimensional submanifold $\Sigma^n \subseteq \R^{n+k}$ we have the bound $\lambda(\Sigma) \ge 1$. The equality is reached if $\Sigma$ is a flat $n$-plane.
\end{remark}

An important feature of the entropy functional is that it is monotonically nonincreasing along a mean curvature flow. This is a consequence of Huisken's monotonicity formula \cite{hu90}. 

\begin{remark}\label{remark_entropy_area_growth}
For any  submanifold $\Sigma^n \subseteq \R^{n+k}$, having finite entropy is equivalent to having bounds on area growth. See for instance Theorem~2.2 in \cite{su18}. In particular there exists a constant $C$ such that for every $x \in \R^{n+k}$ and for every $R > 0$, we have
\begin{equation}\label{equation_area_growth}
\text{Area}\lef \Sigma \cap \mathcal{B}_R(x) \rig \le C \lambda(\Sigma) R^n,
\end{equation}
where $\mathcal{B}_R(x)$ is the open ball in $\R^{n+k}$ of radius $R> 0$ centered at $x$.
\end{remark}

\begin{lemma}\label{lemma_finite_entropy_properness}
Let $\Sigma^n \subseteq \R^{n+k}$ be a complete, noncompact, immersed and oriented submanifold. Let us assume that it has finite entropy $\lambda(\Sigma) < \infty$ and that the mean curvature $H$ is bounded, namely $|H| \le C$ for some constant $C > 0$.

Then $\Sigma$ is properly immersed.
\end{lemma}

This result in particular applies to translating solitons, since they have bounded mean curvature. Note that we do not put any restriction on the codimension $k$. The proof is essentially a corollary of Theorem 2.1 in  \cite{cl98}.

\begin{proof}[Proof of Lemma \ref{lemma_finite_entropy_properness}]
Let $\Sigma^k \subseteq \R^{n+k}$ be a complete, immersed and oriented $k$-dimensional submanifold and let us assume that it is not properly immersed. This implies that there exist $x \in \R^{n+k}$ and a sequence $(p_j)_j \subseteq \Sigma$ such that 
$$
\|p_j - x\|_{\R^{n+k}} \xrightarrow[j \rightarrow \infty]{} x
$$
and such that there exists $\delta > 0$ such that 
$$
\dist^\Sigma(p_j, p_i) \ge 2\delta, \qquad j\ne i,
$$
where $\dist^\Sigma(\cdot, \cdot)$ denotes the intrinsic distance of $\Sigma$.  

Let ${B}^{\Sigma}_\delta(p_j)$ denote the intrinsic geodesic ball of $\Sigma$ of radius $\delta$, centered at $p_j$. 
From Theorem 2.1 in \cite{cl98}, we have that there exists a constant $\beta > 0$ such that 
$$
\mathcal{H}^n\lef {B}^{\Sigma}_\delta(p_j) \rig \ge \beta \delta
$$
for every $j \in \N$, where $\mathcal{H}^n$ denotes the $n$-dimensional Hausdorff measure. Let $\mathcal{B}_R(x)$ be the ball in $\R^{n+k}$ of radius $R> 0$ centered at $x$. Take $R$ large enough such that ${B}^{\Sigma}_\delta(p_j) \subseteq  \mathcal{B}_R(x) $ for every $j$. Then we have
$$
\mathcal{H}^n\lef \Sigma \cap \mathcal{B}_R(x) \rig \ge \sum_{j = 1}^\infty \mathcal{H}^n\lef {B}^{\Sigma}_\delta(p_j)\rig \ge \sum_{j = 1}^\infty \beta \delta =  +\infty.
$$

Therefore 

\begin{align*}
\lambda(\Sigma) &\ge F_{x, 1}(\Sigma) = \frac{1}{(4 \pi t)^{\frac{n}{2}}} \int_{\Sigma} e^{-\frac{\|y - x\|^2}{4}} \, d\mu(y) \\
&\ge \frac{e^{-\frac{R^2}{4}}}{(4 \pi t)^{\frac{n}{2}}} \mathcal{H}^n\lef \Sigma \cap \mathcal{B}_R(x) \rig = +\infty.
\end{align*}
\end{proof}

The entropy of a translater is determined by its asymptotic behavior. More precisely, we have the following explicit way for computing the entropy.
\begin{lemma}
Let $\Sigma^n \subseteq \R^{n+1}$ be a translater with finite entropy. Then 
$$
\lambda(\Sigma) = \lim_{\tau \rightarrow \infty} F_{(0, 1)} \lef \frac{1}{\tau}\Sigma - \tau e_{n+1} \rig.
$$
\end{lemma}

\begin{proof}
Let $(y, t) \in \R^{n+1} \times \R$. From Huisken's monotonicity formula we have that 
\begin{equation}
F_{(y, t) }\lef \Sigma\rig \le F_{(y + \tau e_{n+1}, t + \tau)} \lef \Sigma \rig,
\end{equation}
for any $\tau > 0$ (see equation (1.9) in \cite{cm12} and Lemma 4.2 in \cite{guang16}). Therefore there exists 
$$
\lim_{\tau \rightarrow \infty} F_{(y + \tau e_{n+1}, t+ \tau)} \lef \Sigma\rig \eqqcolon \mu(y, t).
$$
Let $\varepsilon > 0$ and let $(y_0, t_0) \in \R^{n+1} \times \R$ such that $F_{(y_0, t_0)}(\Sigma) \ge \lambda(\Sigma) - \varepsilon$.
Clearly we have that 
\begin{equation}\label{eq_mu}
\lambda(\Sigma) - \varepsilon \le \mu(y_0, t_0) \le \lambda(\Sigma).
\end{equation} 
Moreover it is easy to check that the limit $\mu(y, t)$ actually is a constant, namely it does not depend on $(y, t)$.  Therefore \eqref{eq_mu} implies that $\mu = \lambda(\Sigma)$.
\end{proof}

\begin{corollary}\label{corollary_additivity}
Let $\Sigma^n_1, \Sigma^n_2 \subseteq \R^{n+1}$ be translaters with finite entropy. 

Then 
\begin{equation}\label{entropy_linear_on_transl}
\lambda(\Sigma_1 + \Sigma_2) = \lambda(\Sigma_1) + \lambda(\Sigma_2),
\end{equation}
where ``$\Sigma_1 + \Sigma_2$'' denotes the sum of Radon  measures naturally induced by $\Sigma_1$ and $\Sigma_2$.
\end{corollary}

\begin{remark}
Observe that \eqref{entropy_linear_on_transl} does not hold in general for hypersurfaces which are not translating solitons.
For instance take a hypersurface $\Sigma$ for which the function $(x_0, t_0) \mapsto F_{x_0, t_0}(\Sigma)$ achieves a strict global maximum. This holds true, for instance, for shrinking solitons with polynomial volume growth which do not split off a line isometrically (see Section 7 in \cite{cm12}). Let $V \in \R^{n+1}$ be a nonzero vector and define $\widetilde{\Sigma} \coloneqq \Sigma + V$. Note that the function $(x_0, t_0) \mapsto F_{x_0, t_0}(\Sigma + \widetilde{\Sigma})$ achieves a strict global maximum as well and  $\lambda(\Sigma + \widetilde{\Sigma}) < \lambda(\Sigma) + \lambda(\widetilde{\Sigma})$.
\end{remark}




\begin{thebibliography}{9}

\bibitem[Br16]{br16}
S. Brendle, 
\emph{Embedded self-similar shrinkers of genus $0$},
Ann. of Math. 183 (2), 715-728 (2016).

\bibitem[BLT18]{blt18}
T. Bourni, M. Langford, G. Tinaglia,
\emph{On the existence of translating solutions of mean curvature flow in slab regions},
arXiv:1805.05173v3 (2018).

\bibitem[BGM19]{bgm19}
A. Bueno, J. A. G\' alvez, P. Mira,
\emph{The global geometry of surfaces with
prescribed mean curvature in $\mathbb{R}^3$},
arXiv:1802.08146v2 (2019).

\bibitem[Ch76]{ch76}
S. Y. Cheng,
\emph{Eigenfunctions and nodal sets}, 
Comment. Math. Helv. 51, 43-55 (1976).

\bibitem[CL98]{cl98}
L.-F. Cheung, P.-F. Leung,
\emph{The mean curvature and volume growth of complete
noncompact submanifolds},
 Differential Geom. Appl. 8, no. 3, 251 – 256 (1998).

\bibitem[CM19a]{cm19}
F. Chini, N.M. M\o ller,
\emph{Bi-halfspace and convex hull theorems for translating solitons},
Int. Math. Res. Not.  \url{https://doi.org/10.1093/imrn/rnz183} (2019).

\bibitem[CM19b]{cm19b}
F. Chini, N.M. M\o ller,
\emph{Ancient mean curvature flows and their spacetime tracks},
arXiv:1901.05481v2 (2019).

\bibitem[CSS07]{css07}
J. Clutterbuck, O. Schnürer, F. Schulze,
\emph{Stability of translating solutions to mean curvature flow}, Calc.
Var. Partial Differ. Equ. 29, 281–293 (2007).

\bibitem[CM04]{cm04}
T.H. Colding, W.P. Minicozzi II, 
\emph{The space of embedded minimal surfaces of fixed genus in a 3-manifold II; Multi-valued graphs in disks},
Ann. of MAth., (2) 160, no. 1, 69-92 (2004). 

\bibitem[CM11]{cm-min}
T.H. Colding, W.P. Minicozzi II,
\emph{A Course in Minimal Surfaces},  AMS (2011).

\bibitem[CM12]{cm12}
 T.H. Colding, W.P. Minicozzi II, 
\emph{Generic mean curvature flow I: generic singularities}, 
Ann. of Math. 175, 755-833 (2012).

\bibitem[EH91]{eh91}
K. Ecker, G. Huisken,
\emph{Interior estimates for hypersurfaces moving by mean curvature},
 Invent. Math. 105, no. 3, 547–569 (1991).

\bibitem[Gu16]{guang16}
Q. Guang,
\emph{Volume growth, entropy and stability for translating solitons},
arXiv:1612.05312v1 (2016).

\bibitem[GT83]{gt01}
D. Gilbarg, N. S. Trudinger,
\emph{Elliptic Partial Differential Equations of
Second Order}, Springer-Verlag, 2nd edition, (1983).

\bibitem[Ha15]{h15}
R. Haslhofer,
\emph{Uniqueness of the bowl soliton},
Geom. Topol. Volume 19, Number 4, 2393-2406, (2015).

\bibitem[HHC18]{hhc18}
K. Choi, R. Haslhofer, O. Hershkovits,
\emph{Ancient low entropy flows, mean convex neighborhoods, and uniqueness},
arXiv:1810.08467v1 (2018).

\bibitem[HHCW19]{hhcw19}
K. Choi, R. Haslhofer, O. Hershkovits, B. White,
\emph{Ancient asymptotically cylindrical flows and applications}
arXiv:1910.00639v2 (2019).

\bibitem[HK17]{hk17}
R. Haslhofer, B. Kleiner,
\emph{Mean curvature flow of mean convex hypersurfaces},
Comm. Pure Appl. Math. 70 (3), 511–546 (2017).

\bibitem[He18]{he18}
O. Hershkovits,
\emph{Translators asymptotic to cylinders},
arXiv:1805.10553v1 (2018).

\bibitem[HIMW19a]{himw19a}
D. Hoffman, T. Ilmanen, F. Martín, B. White,
\emph{Graphical translators for mean curvature flow}
 Calc. Var.,  58:117 (2019).

\bibitem[HIMW19b]{himw19b}
D. Hoffman, T. Ilmanen, F. Martín, B. White,
\emph{Notes on translating solitons for mean curvature flow}
arXiv:1901.09101v2 (2019).

\bibitem[HMW19a]{hmw19}
D. Hoffman, F. Mart\'in, B. White,
\emph{Scherk-like translators for mean curvature flow},
arXiv:1903.04617v3 (2019).

\bibitem[HMW19b]{hmw19b}
D. Hoffman, F. Mart\'in, B. White,
\emph{Nguyen's Tridents and the Classification of Semigraphical Translators for Mean Curvature Flow},
arXiv:1909.09241v1 (2019).

\bibitem[Hu90a]{hu90}
G. Huisken,
\emph{Asymptotic behavior for singularities of the mean curvature flow},
J. Differential Geom. 31, 285–299 (1990).

\bibitem[Hu90b]{hu93}
G. Huisken,
\emph{Local and global behaviour of hypersurfaces moving by mean curvature},
Differential Geometry: Partial Differential Equations on Manifolds
(Los Angeles, CA), Proc. Sympos. Pure Math. 54, Amer. Math. Soc. (1990).

\bibitem[Il95]{il95}
T. Ilmanen,
\emph{Singularities of mean curvature flow of surfaces}, preprint (1995).

\bibitem[Il03]{il03}
T. Ilmanen 
\emph{Problems in mean curvature flow}, available at
\url{http://people.math.ethz.ch/~ilmanen/classes/eil03/problems03.ps} (2003).

\bibitem[IR17]{ir17}
D. Impera, M. Rimoldi,
\emph{Rigidity results and topology at infinity of translating solitons of the mean curvature flow},
 Commun. Contemp. Math. 19(6), 21 (2017).

\bibitem[IR19]{ir19}
D. Impera, M. Rimoldi,
\emph{Quantitative index bounds for translators via topology},
arXiv:1804.07709v4 (2019).

\bibitem[KS18]{ks18}
K. Kunikawa, S. Saito,
\emph{Remarks on topology of stable translating solitons},
Geom. Dedicata (2018).

\bibitem[Ma11]{mantegazza}
C. Mantegazza,
\emph{Lecture notes on mean curvature flow}, Birkh\"auser (2011).

\bibitem[MM09]{mm09}
C. Mantegazza, A. Magni, 
\emph{Some remarks on Huisken’s monotonicity formula for mean curvature flow}, 
Singularities in Nonlinear Evolution Phenomena and Applications, CRM Ser. Center “Ennio De Giorgi”, Pisa 9, 157-169, (2009).

\bibitem[MSS15]{mss15}
F. Mart\'in, A. Savas-Halilaj, K. Smoczyk,
\emph{On the topology of translating solitons of the mean curvature flow}, Calculus of Variations and PDE’s, vol. 54(3), 2853 - 2882,  (2015).

\bibitem[M\o14]{mo14}
 N.M. M\o ller, 
\emph{Non-existence for self-translating solitons}, arXiv:1411.2319 (2014).

\bibitem[Ng09]{ng09}
X.H. Nguyen, 
\emph{Translating tridents},
Comm. Partial Differential Equations 34, no. 1-3, 257–280, (2009).

\bibitem[Pé16]{pe16}
J. Pérez-García, 
\emph{Some results on Translating Solitons of the Mean Curvature Flow},
Doctoral thesis, University of Granada (2016).

\bibitem[Ro95]{ro95}
A. Ros, 
\emph{A two-piece property for compact minimal surfaces in a three-sphere},
 Indiana Univ. Math. J. 44, 841–849 (1995).

\bibitem[SS83]{ss83}
R. Schoen, L. Simon,
\emph{Regularity of simply connected surfaces with quasi-conformal Gauss map}, 
Seminar on Minimal Submanifolds, Annals of Math. Studies, vol. 103, 127-145, Princeton University Press, Princeton, N.J. (1983).

\bibitem[Sh15]{sh15}
L. Shahriyari, 
\emph{Translating graphs by mean curvature flow},
Geom. Dedicata, 175(1):57–64, (2015).

\bibitem[SX17]{sx17}
J. Spruck, L Xiao,
\emph{Complete translating solitons to the mean curvature flow in $\R^3$ with nonnegative
mean curvature}, (to appear in  Amer. J. Math)  arXiv:1703.01003v2 (2017).

\bibitem[Su18]{su18}
A. Sun,
\emph{Singularities of mean curvature flow of surfaces with additional forces},
arXiv:1808.03937v1 (2018).

\bibitem[Xi15]{xi15}
Y.L. Xin, 
\emph{Translating solitons of the mean curvature flow}, Calc. Var. 54,  1995–2016 (2015).

\end{thebibliography}
\end{document}